\documentclass{amsart}
\usepackage[top=1in, bottom=1in, left=1.25in, right=1.25in]{geometry}
\usepackage{mathrsfs}
\usepackage{amssymb}
\usepackage{amsfonts}
\usepackage[linkcolor=red,citecolor=blue]{hyperref}
\usepackage{MnSymbol}
\usepackage{latexsym}
\usepackage[all,cmtip]{xy}
\usepackage{amsmath,amsthm}
\usepackage{color,ifsym}
\usepackage{graphicx}
\usepackage{fancybox}
\usepackage{longtable}
\usepackage{lscape}
\usepackage{float}
\usepackage{mathdots}
\usepackage{enumerate}
\usepackage{multirow}

\usepackage{extarrows}

\newtheorem{thm}{Theorem}[section]
\newtheorem{lem}[thm]{Lemma}        
\newtheorem{cor}[thm]{Corollary}
\newtheorem{prop}[thm]{Proposition}

\newtheorem*{abspr*}{Abstract Proposition}
\newtheorem*{abslemma*}{Abstract Lemma}
\newtheorem*{mthm*}{Main Theorem}
\newtheorem*{thm*}{Theorem}
\newtheorem*{defi*}{Definition}
\newtheorem*{lem*}{Lemma}        
\newtheorem*{cor*}{Corollary}
\newtheorem*{prop*}{Proposition}
\newtheorem*{conj*}{Conjecture}
\newtheorem{rem}{{\it Remark}} 

\theoremstyle{plain} 
\newcommand{\thistheoremname}{}
\newtheorem*{genericthm*}{\thistheoremname}
\newenvironment{namedthm*}[1]
{\renewcommand{\thistheoremname}{#1}%
	\begin{genericthm*}}
	{\end{genericthm*}}

\usepackage[leqno]{amsmath}
\makeatletter
\newcommand{\leqnomode}{\tagsleft@true\let\veqno\@@leqno}
\newcommand{\reqnomode}{\tagsleft@false\let\veqno\@@eqno}
\makeatother

\begin{document}

    \title{Rodier type theorem for generalized principal series}
	\author{CAIHUA LUO}
	\address{Department of Mathematical Sciences, Chalmers University of Technology and the University of Gothenburg, Chalmers Tv\"{a}rgata 3, SE-412 96 G\"{o}teborg}
	\email{caihua@chalmers.se}
	\date{}
	\maketitle
	
%
	
	\begin{abstract}
		Given a regular supercuspidal representation $\rho$ of the Levi subgroup $M$ of a standard parabolic subgroup $P=MN$ in a connected reductive group $G$ defined over a non-archimedean local field $F$, we serve you a Rodier type structure theorem which provides us a geometrical parametrization of the set $JH(Ind^G_P(\rho))$ of Jordan--H{\"o}lder constituents of the Harish-Chandra parabolic induction representation $Ind^G_P(\rho)$, vastly generalizing Rodier structure theorem for $P=B=TU$ Borel subgroup of a connected split reductive group about 40 years ago. Our novel contribution is to overcome the essential difficulty that the relative Weyl group $W_M=N_G(M)/M$ is not a coxeter group in general, as opposed to the well-known fact that the Weyl group $W_T=N_G(T)/T$ is a coxeter group. Indeed, such a beautiful structure theorem also holds for finite central covering groups.
	\end{abstract}

\section{introduction}
Following Harish-Chandra's ``philosophy of cusp forms'' which  culminates in the Langlands classification theorems under parabolic induction both locally and globally, a longstanding local problem is to understand the decomposition structures of parabolic inductions, especially those inducing from supercuspidal representations which are the so-called generalized principal series. Among those parabolic inductions, there are two extreme cases, namely unitary parabolic inductions and regular generalized principal series of which many mathematicians have devoted their efforts to describe the corresponding internal structures. To be more precise, 
\begin{itemize}
	\item \emph{Tempered parabolic induction}: Knapp--Stein R-group theory and its explicit structures (Bruhat, Harish-Chandra, Knapp--Stein, Jacquet, Casselman, Howe, Silberger, Winarsky, Keys, Shahidi, Goldberg etc).
	\item \emph{Principal series}: a. Muller's irreducibility criterion for principal series \cite{muller1979integrales}; b. Rodier's structure theorem for regular principal series \cite{rodier1981decomposition}.
\end{itemize}
So one might ask the following natural questions:
\begin{align*}
Q1:~&\mbox{What is the irreducibility criterion for generalized principal series in terms of Muller? }\\
Q2:~&\mbox{What is the story of Rodier structure theorem for regular generalized principal series?}
\end{align*}
The first question in principal should be doable after Muller's work, but we have not seen any literature and will write down the details separately (cf. \cite{luo2018muller}). As for the second question, without the far-reaching Langlands--Shahidi theory built up by Shahidi in the early 1990s (cf. \cite{shahidi1990proof}), it seems that one cannot push Rodier's theorem further to regular generalized principal series if following Rodier's paper completely, especially the argument of Proposition 3. On the other hand, for general parabolic subgroup $P=MN\supset B=TU$ in $G$, it is well-known that the relative Weyl group $W_M=N_G(M)/M$ is not a coxeter group, as opposed to the Weyl group $W_T=N_G(T)/T$. To overcome those difficulties, we do some general observations which enable us to deal with generalized principal series in the way of dealing with principal series of split group. To be precise, let $X(M)_F$ be the group of $F$-rational characters of $M$, we set $\mathfrak{a}^\star_M:=X(M)_F\otimes_\mathbb{Z}\mathbb{R}$.
We denote by $\Phi_M$ the set of reduced relative roots of $M$ in $G$, by $\Delta_M$ the set of relative simple roots determined by $N$. Let $\Phi^0_M$ be the set of those relative roots which contribute reflections in $W_M$. Denote by $W_M^1:=\big\{w\in W_M:~w.(\Phi_M^0)^+>0 \big\}$ and by $W_M^0$ the ``small'' relative Weyl group, i.e. $W_M^0:=\left<w_\alpha:~\alpha\in \Phi^0_M \right>.$ Define $Ind_P^G(\rho)$ to be the generalized principal series inducing from the parabolic subgroup $P=MN$ to $G$ with $\rho$ a supercuspidal representation of $M$. If $\rho 
$ is regular, i.e. $W_\rho:=\big\{w\in W_M:~w.\rho=\rho \big\}=\big\{1\big\}$, we say that $Ind^G_P(\rho)$ is a regular generalized principal series. Concerning the fact that $W_M$ is not a coxeter group in general, we (re)discover two key observations as follows:
\begin{enumerate}[(i)]
	\item (see Lemma \ref{key2}) \[W_M=W_M^0\rtimes W_M^1. \]
	\item (see Lemma \ref{key1})) For $w\in W_M^0$ and $w_1\in W_M^1$, we have, 
	\[Ind^G_P(\rho^w)\simeq Ind^G_P(\rho^{ww_1}). \]
\end{enumerate}
Very recently, we learned that the first observation is an old result of Lusztig and has since been applied creatively in \cite[Lemma 5.2]{lusztig1976coxeter}, \cite[Corollary 2.3]{howlett1980induced}, \cite{morris1993tamely} and \cite{brink1999normalizers}. Another observation concerns the uniformity of the location of the unique pole of co-rank one Plancherel measures, please see 
Theorem \ref{linearindep} for details (cf. \cite{silberger1980special} for the uniqueness claim). For regular supercuspidal representation $\rho$ of $M$, we denote by $S$ the set of those positive relative coroots $\alpha^\vee$ with $\alpha\in \Phi_M^0$ such that the co-rank one Plancherel measure $\mu_\alpha(\cdot)$ has a pole at $\rho$. Denote by $^{0}\mathfrak{a}_M^*$ the relative ``small'' root space
\[^{0}\mathfrak{a}^\star_M:=Span_\mathbb{R}\{\alpha:~\alpha\in \Phi_M^0 \},\]
and by $C^+_M$ the relative positive dominant Weyl chamber in $^{0}\mathfrak{a}_M^*$ determined by $P$.
 
In view of those structures, Rodier's argument in \cite{rodier1981decomposition} applies seamlessly to generalized principal series which says that
\begin{thm}(Rodier Type Structure Theorem i.e. Theorem \ref{mthm})
	The constituents $\pi_\Gamma$ of the regular generalized principal series $Ind^G_P(\rho)$ are parameterized by the connected components $\Gamma$ of $$^{0}\mathfrak{a}_M^\star-\bigcup_{\alpha^\vee\in S}Ker(\alpha^\vee)$$
	satisfying the following property:
	
	the Jacquet module $r_P(\pi_\Gamma)$ of $\pi_\Gamma$ with respect to $P$ is equal to 
	\[\bigoplus_{wC^+_M\subset \Gamma}\rho^w. \]	
\end{thm}
As an essential input for the determination of the square-integrable/tempered constituents of $Ind^G_P(\rho)$, we need the following key claim:
\begin{thm}(cf. Theorem \ref{linearindep})
	Keep the notions as before. The set $S$ is linearly independent.
\end{thm}
Denote by $\omega_\rho$ the real unramified character of $M$, i.e. $\omega_\rho\in \mathfrak{a}_M^*$, such that the central character of $\omega_\rho^{-1}\rho$ is unitary. Let $^{+}\mathfrak{a}^\star_M ~(resp.~ ^{+}\bar{\mathfrak{a}}^\star_M)$ be the set of such $\chi\in \mathfrak{a}_M^\star$ of the form
\[\chi=\sum\limits_{\alpha\in \Delta_M}x_\alpha \alpha,\]
with all the coefficients $x_\alpha>0~(resp.~x_\alpha\geq 0)$. Denote by $\mathfrak{a}_M^{*+}$ $(resp. \bar{\mathfrak{a}}_M^{*+})$ the (resp. closure of) dominant Weyl chamber in $\mathfrak{a}_M^*$ determined by $\Delta_M$. Without loss of generality, we assume that $\omega_\rho\in \bar{\mathfrak{a}}_M^{*+}$ and write $\Gamma_+=\bigcap\limits_{\alpha^\vee\in S}(\alpha^\vee)^{-1}(\mathbb{R}^+)$. As applications of those theorems, we first prove a sufficient and necessary condition for the existence of square-integrable/tempered constituents as follows:
\begin{prop}(cf. Proposition \ref{ds})
	Keep the notions as before. $JH(Ind^G_P(\rho))$ contains at most one square-integrable constituent. Moreover
	the representation $\pi_\Gamma$ is square-integrable if and only if $^{0}\mathfrak{a}_M^*=\mathfrak{a}_M^*=Span_{\mathbb{R}}\{\alpha:~\alpha^\vee\in S \}$ and $\Gamma=\Gamma_+$.
\end{prop}
\begin{prop}(cf. Proposition \ref{tem})
	Keep the notions as before. $JH(Ind^G_P(\rho))$ contains at most one tempered constituent. Moreover the representation $\pi_\Gamma$ is tempered if and only if $\omega_\rho$ restricting to the subgroup $\bigcap\limits_{\alpha^\vee\in S}Ker(\alpha^\vee)$ of $M$ is unitary and $\Gamma=\Gamma_+$.
\end{prop}
Along the way, we also provide a simple proof of Casselman--Shahidi's main theorem \cite[Theorem 1]{casselman1998irreducibility} which says that
\begin{thm}(cf. Theorem \ref{generic})
	For the standard representation $Ind^G_P(\rho)$ with $\rho$ a generic supercuspidal representation of the Levi subgroup $M$ of a standard parabolic subgroup $P=MN$ in a connected quasi-split reductive group $G$, we have that 
	\[\mbox{ the unique generic subquotient of }Ind^G_P(\rho) \mbox{ is a subrepresentation.} \] 
\end{thm}
At last, we would like to mention that Tadi{\'c} has worked out Rodier type structure theorem for quasi-split groups $GSp_{2n}, Sp_{2n}$ and $SO_{2n+1}$ for which he does not have to overcome any difficulties as mentioned above (cf. \cite{tadic1998regular}).

In the end, we give the outline of the paper. In Section 2, some necessary notions of representation theory are introduced. In Section 3, we will first state and prove the Rodier type structure theorem for regular generalized principal series which roughly speaking is a parametrization of the associated Jordan--H\"{o}lder set $JH(Ind^G_P(\rho))$. In Section 4, we prove the linear independence property of the set $S$. Section 5 is about the parametrization/characterization of the constituents of discreteness/temperedness of generalized principal series, while the determination of the generic constituents will be discussed in Section 6. The last section is to see how Aubert duality behaves under our setting.

	\paragraph*{\textbf{Acknowledgements}}
	I are much indebted to Professor Wee Teck Gan for his guidance and numerous discussions on various topics. I would like to thank Martin Raum for his kindness and support. I would like to thank Professor Marko Tadi\'{c} for patiently answering my questions, and thank Max Gurevich for his seminar talk on a conjectural criterion of the irreducibility of parabolic inductions for $GL_n$ in the National University of Singapore which rekindles our enthusiasm to explore the mysterious internal structures of parabolic inductions. I would also like to thank Professor Chengbo Zhu for his financial support during my last month stay in Singapore.
\section{preliminaries}
\subsection{Weyl group}Let $G$ be a connected reductive group defined over a non-archimedean local field $F$ of characteristic 0. For our purpose, it is no harm to assume that the center $Z_G$ of $G$ is compact. Denote by $|-|_F$ the absolute value, by $\mathfrak{w}$ the uniformizer and by $q$ the cardinality of the residue field of $F$. Fix a Borel subgroup $B=TU$ of $G$ with $T$ a minimal Levi subgroup and $U$ a maximal unipotent subgroup of $G$, and let $P=MN$ be a standard parabolic subgroup of $G$ with $M$ the Levi subgroup and $N$ the unipotent radical.

Let $X(M)_F$ be the group of $F$-rational characters of $M$, and set 
\[\mathfrak{a}_M=Hom(X(M)_F,\mathbb{R}),\qquad\mathfrak{a}^\star_{M,\mathbb{C}}=\mathfrak{a}^\star_M\otimes_\mathbb{R} \mathbb{C}, \]
where
\[\mathfrak{a}^\star_M=X(M)_F\otimes_\mathbb{Z}\mathbb{R} \]
denotes the dual of $\mathfrak{a}_M$. Recall that the Harish-Chandra homomorphism $H_P:M\longrightarrow\mathfrak{a}_M$ is defined by
\[q^{\left< \chi,H_P(m)\right>}=|\chi(m)|_F \] 
for all $\chi\in X(M)_F$.

Next, let $\Phi$ be the root system of $G$ with respect to $T$, and $\Phi^+$ $(resp.~\Delta)$ be the set of positive (resp. simple) roots determined by $U$. For $\alpha\in \Phi$, we denote by $\alpha^\vee$ the associated coroot, and by $w_\alpha$ the associated reflection in the Weyl group $W$ of $T$ in $G$ with
\[W=W_T:=N_G(T)/T=\left<w_\alpha:~\alpha\in\Phi\right>. \]
The \underline{walls} in $\mathfrak{a}_T^\star$ are the hyperplanes $Ker~\alpha^\vee$. The \underline{Weyl chambers} in $\mathfrak{a}^\star_T$ are the connected components of the set
\[\mathfrak{a}^\star_T-\bigcup_{\alpha\in \Phi^+}Ker~\alpha^\vee  \]
on which the Weyl group $W$ acts simply transitively (cf. \cite{springer2010linear}). We denote by $C^+$ the dominant Weyl chamber determined by $B$. Denote by $w_0^G$ the longest Weyl element in $W$, and similarly by $w_0^M$ the longest Weyl element in the Weyl group $W^M$ of a Levi subgroup $M$. 
 
\subsection{Relative Weyl group}Likewise, we denote by $\Phi_M$ (resp. $\Phi_M^+$) the set of reduced relative (resp. positive) roots of $M$ in $G$, by $\Delta_M$ the set of relative simple roots determined by $N$ and by $W_M:=N_G(M)/M$ the relative Weyl group of $M$ in $G$. In general, a relative reflection $\omega_\alpha:=w_0^{M_\alpha}w_0^M$ with respect to a relative root $\alpha$ does not preserve our Levi subgroup $M$. Denote by $\Phi^0_M$ (resp. $(\Phi_M^0)^+$) the set of those relative (resp. positive) roots which contribute reflections in $W_M$. It is easy to see that $W_M$ preserves $\Phi_M$, and further $\Phi_M^0$ as well, as $\omega_{w.\alpha}=w\omega_\alpha w^{-1}$. Note that $W_M$ in general is larger than the one generated by those relative reflections, for example,
\[M\simeq GL_1\times GL_3\subset SO_8, \]
where $W_M\simeq \mathbb{Z}/2\mathbb{Z}$, while there are no relative reflections preserving $M$. For other parabolic subgroup $P'=MN'$, we denote $\Phi_M(P')$ to be the set of reduced relative roots determined by $N'$. In particular, if $P'=P$, then $\Phi_M(P)=\Phi_M^+$.

For our purpose, we define the ``small'' relative Weyl group $W_M^0\subset W_M$ to be the one generated by those relative reflections, i.e.
\[W_M^0:=\left<w_\alpha:~\alpha\in \Phi^0_M \right>. \]
It is easy to see that, as $w\omega_{\alpha}w^{-1}=\omega_{w.\alpha}$ for $\alpha\in \Phi_M^0$ and $w\in W_M$, 
\[W_M^0\lhd W_M. \]
The relative \underline{walls} in the ``small'' character vector spaces $$^{0}\mathfrak{a}^\star_M:=Span_\mathbb{R}\{\alpha:~\alpha\in \Phi_M^0 \}$$ are those hyperplanes $Ker~\alpha^\vee$ which contribute reflections in $W_M$. The relative \underline{Weyl chambers} are the connected components of the set
\[^{0}\mathfrak{a}^\star_M-\bigcup_{\alpha\in (\Phi^0_M)^+} Ker~\alpha^\vee\]
on which the ``small'' relative Weyl group $W_M^0$ acts. An observation is that 
\[\mbox{$W_M^0$ acts simply transitively on the set of relative Weyl chambers,}\]
which follows from the fact that $\Phi_M^0$ is a root system, may not be irreducible. Denote by $\Delta_M^0$ the relative simple roots of $\Phi_M^0$. Note that $\Phi_M^0$ is quite different with $\Phi_M$ in general, as well as $\Delta_M^0$ and $\Delta_M$, for example,
\[M\simeq GL_2\times GL_4\subset SO_{12}, \]
where $\Delta_M=\big\{e_1-e_2,e_2\big\}\subset \Phi_M=\big\{\pm e_1\pm e_2,\pm e_1,\pm e_2 \big\}$, while $\Delta_M^0=\big\{e_1,e_2 \big\}\subset \Phi_m^0=\big\{\pm e_1,\pm e_2 \big\}.$

We denote by $C^+_M$ the relative dominant Weyl chamber in $^{0}\mathfrak{a}_M^\star$ determined by $P$. Recall that the canonical pairing $$\left<-,-\right>:~\mathfrak{a}^\star_M\times \mathfrak{a}_M\longrightarrow\mathbb{Z}$$ suggests that each  $\alpha\in \Phi_M$ will enjoy a one parameter subgroup $H_{\alpha^\vee}(F^\times)$ of $M$ satisfying: for $x\in F^\times$ and $\beta\in \mathfrak{a}^\star_M$,
\[\beta_\alpha(x):=\beta(H_{\alpha^\vee}(x))=x^{\left<\beta,\alpha^\vee\right>}. \]

\subsection{Parabolic induction and Jacquet module}For $P=MN$ a parabolic subgroup of $G$ and an admissible representation $(\sigma,V_\sigma)~ (resp.~(\pi,V_\pi))$ of $M~(resp.~G)$, we have the following normalized parabolic induction of $P$ to $G$ which is a representation of $G$
\[Ind_P^G(\sigma):=\big\{\mbox{smooth }f:G\rightarrow V_\sigma|~f(nmg)=\delta_P(m)^{1/2}\sigma(m)f(g), \forall n\in N, m\in M~and~g\in G\big\} \]
with $\delta_P$ stands for the modulus character of $P$, i.e., denote by $\mathfrak{n}$ the Lie algebra of $N$,
\[\delta_P(nm)=|det~Ad_\mathfrak{n}(m)|_F, \]
and the normalized Jacquet module $r_P(\pi)$ with respect to $P$ which is a representation of $M$
\[\pi_N:=V_\pi/\left<\pi(n)e-e:~n\in N,e\in V_\pi\right>. \] 

\subsection{Co-rank one reducibility}For $\alpha\in \Phi_M$, let $M_\alpha\supset M$ be the co-rank one Levi subgroup determined by $\alpha$, co-rank one reducibility means the reducibility of those $Ind^{M_\alpha}_{P\cap M_\alpha}(\sigma)$ for $\alpha\in \Phi_M$. Notice that $Ind^{M_\alpha}_{P\cap M_\alpha}(\sigma)$ is always irreducible if $\alpha\notin \Phi^0_M$ (cf. \cite{silberger1980special}), so we will only talk about the co-rank one reducibility associated to those $\alpha\in \Phi^0_M$ in the paper.

\subsection{Whittaker model}For this purpose, we shall assume that $G$ is quasi-split. For each root $\alpha\in\Phi$, there exists a non-trivial homomorphism $X_\alpha$ of $F$ into $G$ such that, for $t\in T$ and $x\in F$, 
\[tX_\alpha(x)t^{-1}=X_\alpha(\alpha(t)x). \]
We say a character $\theta$ of $U$ is \underline{generic} if the restriction of $\theta$ to $X_\alpha(F)$ is non-trivial for each simple root $\alpha\in \Delta$. Then the Whittaker function space $\mathcal{W}_\theta$ of $G$ with respect to $\theta$ is the space of smooth complex functions $f$ on $G$ satisfying, for $u\in U$ and $g\in G$,
\[f(ug)=\theta(u)f(g), \]
i.e. $\mathcal{W}_\theta=Ind_U^G(\theta)$. We say an irreducible admissible representation $\pi$ of $G$ is \underline{$\theta$-generic} if
\[\pi\xrightarrow[non-trivial]{G-equiv.} \mathcal{W}_\theta. \]

\subsection{Square-integrability/Temperedness}In what follows, we recall Casselman's square-integrability and temperedness criterion. For our purpose, we only state it under the condition that the inducing datum $\rho$ is supercuspidal, i.e. $\pi\in JH(Ind^G_P(\rho))$ with $\rho$ supercuspidal representation of $M$, here $JH(-)$ means the set of Jordan--H\"{o}lder constituents.

Let $^{+}\mathfrak{a}^\star_M ~(resp.~ ^{+}\bar{\mathfrak{a}}^\star_M)$ be the set of such $\chi\in \mathfrak{a}_M^\star$ of the form
\[\chi=\sum\limits_{\alpha\in \Delta_M}x_\alpha \alpha,\]
with all the coefficients $x_\alpha>0~(resp.~x_\alpha\geq 0)$.
Denote by $\mathfrak{a}_M^{*+}$ $(resp. \bar{\mathfrak{a}}_M^{*+})$ the (resp. closure of) dominant Weyl chamber in $\mathfrak{a}_M^*$ determined by $\Delta_M$. For an admissible representation $(\tau,V_\tau)$ of $M$, we define the set $\mathcal{E}xp(\tau)$ of exponents of $\tau$ as follows:
\[\mathcal{E}xp(\tau)=\big\{\chi\in\mathfrak{a}_{M,\mathbb{C}}^\star:~V_\tau[\chi]\neq 0 \big\}, \]
where $Z_M$ is the center of $M$ and 
\[V_\tau[\chi]=\big\{e\in V:~\exists d\in \mathbb{N}, \forall t\in Z_M, (\tau(t)-\chi(t))^d e=0 \big\}. \] 
Keep the notation as above, $\pi\in JH(I(\nu,\sigma))$ is square-integrable (resp. tempered) if and only if for each standard parabolic subgroup $Q=LV$ associated to $P=MN$, i.e. $L$ and $M$ are conjugate,  \[Re(\mathcal{E}xp(\pi_V))\subset~^{+}\mathfrak{a}^\star_M~ (resp.~ ^{+}\bar{\mathfrak{a}}^\star_M).\]

\section{rodier type structure theorem of generalized principal series}
Recall that an irreducible supercuspidal representation $\tau$ of $M$ is called \underline{regular} in $G$ if the only element $w\in W_M$ such that $\tau^w\simeq \tau$ is the identity element. A parabolic induction $Ind^G_{P=MN}(\rho)$ is called a generalized principal series if the inducing data $\rho$ is a supercuspidal representation of the Levi subgroup $M$ of $P$. Furthermore, if our inducing data is a regular supercuspidal representation, we call the associated induced representation a regular generalized principal series. In this section, we will first recall Rodier's structure theorem of the constituents of regular principal series of split groups (see \cite[Theorem, Pg.418]{rodier1981decomposition}), then extend it to regular generalized principal series of arbitrary connected reductive group and its finite central covering group.

Recall that in \cite{rodier1981decomposition}, for a regular character $\chi$ of the torus $T$ of the Borel subgroup $B=TU$ of a connected split reductive group $G$, let $S$ be the set of coroots $\alpha^\vee$ such that $\chi_\alpha=|\cdot|$, and $-S:=\{-\alpha^\vee:~\alpha\in S \}$. Then
\begin{thm}[Rodier structure theorem] (see \cite[Theorem, Pg. 418]{rodier1981decomposition})\label{rps}
	
	The constituents $\pi_\Gamma$ of the regular principal series $Ind^G_B(\chi)$ are parameterized by the connected components $\Gamma$ of $$\mathfrak{a}_T^\star-\bigcup_{\alpha^\vee\in S}Ker(\alpha^\vee)$$
	satisfying the following property:
	
	the Jacquet module $r_B(\pi_\Gamma)$ of $\pi_\Gamma$ with respect to $B$ is equal to 
	\[\bigoplus_{wC^+\subset \Gamma}\chi^w. \]
\end{thm} 
Before turning to our Rodier type structure theorem for regular generalized principal series of arbitrary connected reductive groups, we first investigate the main ideas behind Rodier structure theorem which has not been pointed out explicitly in \cite{rodier1981decomposition}. Once those ideas are streamlined clearly, it would be readily to see how simple yet beautiful our Rodier type structure theorem is. 

For regular principal series $Ind^G_B(\chi)$, we have that, as representations of $T$, 
\[r_B(Ind^G_B(\chi))=r_B(Ind^G_B(\chi^w))=\bigoplus_{w'\in W_T}\chi^{w'}, \]
for any $w\in W_T$. Applying Frobenius reciprocity, we know that 
$Ind^G_B(\chi^w)$ has a unique irreducible subrepresentation, and
$$Hom_G(Ind^G_B(\chi^w),~Ind^G_B(\chi^{w'}))\simeq \mathbb{C}$$
for any $w,~w'\in W_T$.

As the Jacquet module functor is exact (cf. \cite{bernstein1977induced,casselman1995introduction,waldspurger2003formule}), so for any 
$\pi\in JH(Ind^G_B(\chi))$ and any $w\in W_T$, we know that $\pi$ is of multiplicity at most one in $Ind^G_B(\chi^w)$ and is uniquely determined by its Jacquet module $r_B(\pi)$ with respect to $B$. Moreover, for any $\chi^w\in r_B(\pi)$,
\[Ind^G_B(\chi^w)\simeq Ind^G_B(\chi^{w'})\mbox{ if and only if }\chi^{w'}\in r_B(\pi). \] 
Therefore, the determination of the set $JH(Ind^G_B(\chi))$ of the constituents of $Ind^G_B(\chi)$ is equivalent to determining the orbits $\mathcal{O}$ of the set $\{\chi^w:~w\in W_T \}$ under the equivalent relation $\sim$: for $w,~w'\in W_T$,
\[\chi^w\sim \chi^{w'} \mbox{ if and only if } A(w,w')\mbox{ is an isomorphism,} \]
where $A(w,w')$ is the unique, up to scalar, non-zero $G$-equivalent homomorphism in $$Hom_G(Ind^G_B(\chi^w),~ Ind^G_B(\chi^{w'}))\simeq \mathbb{C}.$$
For simplicity, we will abbreviate those intertwining operators $A(w,w')$ as $A$ in what follows.

The next step is to give a characterization of those pairs $(w,w')\subset W_T$ satisfying
\[Ind^G_B(\chi^w)\simeq Ind^G_B(\chi^{w'}). \]
Let us first take a look at the simple basic case, i.e. the pairs $(w,w')$ with $w'=ww_\alpha$ for some simple root $\alpha$. That is to say
\[wC^+\mbox{ and }w'C^+\mbox{ share the same wall }Ker(w.\alpha^\vee). \]
For such a pair, via the induction by stage property of parabolic inductions and the uniqueness property of the intertwining operators $A$, we have also the induction by stage property of $A$, i.e. the following diagram commutes:
\[\xymatrix{Ind^G_{B}(\chi^{w})\ar[r]^A\ar@{=}[d]&Ind^G_{B}(\chi^{ww_{\alpha}})\ar@{=}[d]\\
Ind^G_{P_\alpha}\circ Ind^{M_\alpha }_{B\cap M_\alpha}(\chi^{w})\ar[r]^-{Ind(A)}&Ind^G_{P_\alpha}\circ Ind^{M_\alpha }_{B\cap M_\alpha}(\chi^{ww_{\alpha}}),       } \]
where $P_\alpha=M_\alpha N_\alpha$ is the co-rank one parabolic subgroup associated to the simple root $\alpha$ (cf. \cite{silberger2015introduction}). Whence
\[Ind^G_B(\chi^w)\simeq Ind^G_B(\chi^{ww_\alpha})  \]
if and only if
\[Ind^{M_\alpha}_{B\cap M_\alpha}(\chi^{w})\simeq Ind^{M_\alpha }_{B\cap M_\alpha}(\chi^{ww_{\alpha}})\]
if and only if
\[(\chi^w)_\alpha=\chi_{w.\alpha}\neq |\cdot|^{\pm 1}.\tag*{($\star$)} \]
Moreover, if $(\star)$ does not hold, i.e. $\chi_{w.\alpha}=|\cdot|^{\pm 1}$, which is to say that $Ind^{M_\alpha}_{B\cap M_\alpha}(\chi^w)$ is reducible, then the Jacquet modules of $Ker(A)$ and $Im(A)$ with respect to $B$ are as follows:
\[r_B(Ker(A))=r_B\circ Ind^G_{P_\alpha}(Ker(A))=\bigoplus_{w''\in W_T:~w''^{-1}.\alpha>0}\chi^{ww''}\tag*{$(\star\star)$} \]
and
\[r_B(Im(A))=r_B\circ Ind^G_{P_\alpha}(Im(A))=\bigoplus_{w''\in W_T:~w''^{-1}.\alpha>0}\chi^{ww_\alpha w''}, \]
which follow from Bernstein--Zelevinsky geometrical lemma (see \cite{bernstein1977induced,casselman1995introduction,waldspurger2003formule}). Notice that any coroot $\beta'^\vee$ corresponding to a positive root $\beta'$ satisfies the following relations:
\[\left<w''.\beta,~w''.\beta'^\vee\right>=\left<\beta,~\beta'^\vee\right>>0\mbox{ for any }\beta\in C^+\mbox{ and any }w''\in W_T. \]
In particular,
\[\left<\beta,~\alpha^\vee \right>>0\mbox{ and }\left<\beta,~w''^{-1}.\alpha^\vee \right>=\left<w''.\beta,~\alpha^\vee \right>>0\mbox{ for any }\beta\in C^+\mbox{ and any } w''\in W_T\mbox{ s.t. }w''^{-1}.\alpha>0. \]
Thus we obtain a geometrical characterization of the double coset \[ P_\alpha\backslash G/B=\big\{w''\in W_T:~w''^{-1}.\alpha>0 \big\}=\big\{w''\in W_T:~w''C^+\mbox{ and }C^+ \mbox{ are on the same side of }Ker(\alpha^\vee) \big\}. \]
Applying the same argument as above after conjugating $C^+,~\alpha$ and $w''$ by $w$, we know that the Jacquet module $r_B(Ker(A))$ of $Ker(A)$ in $(\star\star)$ is equal to
\[r_B(Ker(A))=\bigoplus_{w''}\chi^{w''}, \]
where $w''$ runs over those Weyl elements in $W_T$ such that 
\[wC^+\mbox{ and }w''C^+\mbox{ are on the same side of }Ker(w.\alpha^\vee). \]
Analogously, 
\[r_B(Im(A))=\bigoplus_{w''}\chi^{w''},\tag{$SB$} \]
where $w''$ runs over those Weyl elements in $W_T$ such that 
\[w'C^+\mbox{ and }w''C^+\mbox{ are on the same side of }Ker(w.\alpha^\vee). \]
At last, we would like to emphasize that the condition $(\star)$ is equivalent to saying that
\[\mbox{the co-rank one parabolic induction associated to $w.\alpha$ is irreducible.} \]
Such a condition is the right language we need in our statement of Rodier type structure theorem later on.

Now we turn to the discussion of the general case, i.e. the pairs $(w,w')$ with $w'=ww_{\alpha_1}\cdots w_{\alpha_s}$ for some simple roots $\alpha_1,\cdots,\alpha_s$. We require that such a decomposition is minimal in the sense that it gives rise to a minimal gallery between $wC^+$ and $w'C^+$ (see \cite[Section 1.2]{casselman1995introduction}), i.e. $w^{-1}w'=w_{\alpha_1}\cdots w_{\alpha_s}$ is a reduced decomposition. Thus we have (see \cite[Lemma 8.3.2]{springer2010linear})
\[R(w^{-1}w'):=\big\{\alpha\in \Phi^+:~(w^{-1}w').\alpha<0 \big\}=\big\{\alpha_s, ~w_{\alpha_s}.\alpha_{s-1},~\cdots, ~(w_{\alpha_s}\cdots w_{\alpha_2}).\alpha_1 \big\}.\tag{$RD$} \]
The key observation for the general case is to show that 
\[A(w,w')=A(ww_{\alpha_1}\cdots w_{\alpha_{s-1}},w')\circ\cdots\circ A(ww_{\alpha_1},ww_{\alpha_1}w_{\alpha_2})\circ A(w,ww_{\alpha_1}),\tag{$KO$} \]
where $A(w_1,w_2)$ is the unique, up to scalar, nonzero intertwining operator in $$Hom_G(Ind^G_B(\chi^{w_1}),~Ind^G_B(\chi^{w_2}))\simeq \mathbb{C},$$ for any $w_1,~w_2\in W_T$.

Before moving to the proof of the claim $(KO)$, we first discuss its implication to our previous question, i.e. when is $A(w,w')$ an isomorphism? 

Given $(KO)$, it is easy to see that $A(w,w')$ is an isomorphism if and only if
\[\big\{w.\alpha_1^\vee, ww_{\alpha_1}.\alpha_2^\vee,\cdots,ww_{\alpha_1}\cdots w_{\alpha_{s-1}}.\alpha_s^\vee \big\}\bigcap (S\bigcup -S)=\emptyset. \]
Thus Rodier structure theorem, i.e. Theorem \ref{rps} holds.

In what follows, we finish the proof of our claim $(KO)$. To show the equality $(KO)$ is to show that the composition map on the right hand side is nonzero, i.e. the Jacquet module of its image is nonzero. Notice that the Jacquet module of its image on the right hand side of $(KO)$ with respect to $B$ is equal to
\[Jim:=r_B(Im(A(w,ww_{\alpha_1})))\bigcap r_B(Im(A(ww_{\alpha_1},ww_{\alpha_1}w_{\alpha_2})))\bigcap\cdots\bigcap r_B(Im(A(ww_{\alpha_1}\cdots w_{\alpha_{s-1}},w'))), \]
which follows from the simple fact that Jacquet module is invariant under isomorphism. Then it reduces to show that $Jim\neq 0$. Recall that $(SB)$ says:
\begin{align*}
r_B&(Im(A(w,ww_{\alpha_1})))=\\
&\big\{w''\in W_T:~ww_{\alpha_1}C^+\mbox{ and }w''C^+\mbox{ are on the same side of }Ker(w.\alpha_1^\vee) \big\}, \\
r_B&(Im(A(ww_{\alpha_1},ww_{\alpha_1}w_{\alpha_2})))=\\
&\big\{w''\in W_T:~ww_{\alpha_1}w_{\alpha_2}C^+\mbox{ and }w''C^+\mbox{ are on the same side of }Ker(ww_{\alpha_1}.\alpha_2^\vee) \big\},\\
&\vdots\\
r_B&(Im(A(ww_{\alpha_1}\dots w_{\alpha_{s-1}},w')))=\\
&\big\{w''\in W_T:~w'C^+\mbox{ and }w''C^+\mbox{ are on the same side of }Ker(ww_{\alpha_1}\dots w_{\alpha_{s-1}}.\alpha_s^\vee) \big\}.  
\end{align*}
From the expression of $R(w^{-1}w')$ in $(RD)$, it is readily to check that    
\[\chi^{w'}\in Jim, \]
whence the claim $(KO)$ holds.

At last, we would like to see what explicit form of the set $Jim$ is for its own sake. As $$Jim\bigcup Jer=\big\{w''\in W_T:~\chi^{w''} \big\}\mbox{ and }Jim\bigcap Jer=\emptyset$$ where $Jer:=$  
\[r_B(Ker(A(w,ww_{\alpha_1})))\bigcup r_B(Ker(A(ww_{\alpha_1},ww_{\alpha_1}w_{\alpha_2})))\bigcup\cdots\bigcup r_B(Ker(A(ww_{\alpha_1}\cdots w_{\alpha_{s-1}},w'))), \]
it is equivalent to determine the set $Jer\overset{(KO)}{=}r_B(Ker(A(w,w')))$.
From the geometrical description of $r_B(Ker(A))$ in $(SB)$, we obtain that
\[Jer=r_B(Ker(A(w,w')))=\bigoplus_{w''\in Y}\chi^{w''} \]
where $Y$ is the set of $w''\in W_T$ for which there exists a coroot $\alpha^\vee\in S$ such that the chambers $wC^+$ and $w''C^+$ are on the same side of the wall $Ker(\alpha^\vee)$, and the chambers $wC^+$ and $w'C^+$ are separated by the wall. Moreover,
\[Jim=r_B(Im(A(w,w')))=\bigoplus_{w''\in W_T-Y}\chi^{w''}.  \]

Now we turn to the discussion of extending Rodier structure theorem, i.e. Theorem \ref{rps} to regular generalized principal series for arbitrary connected reductive group and its finite central covering group. Recall that the key ideas in the proof of Rodier structure theorem for regular principal series analyzed as above are as follows:
\begin{enumerate}[(i)]
	\item The double coset $B\backslash G/B=W_T=\left<w_\alpha:~\alpha\in \Delta\right>$ is a coxeter group.
	\item For a reduced decomposition of $w^{-1}w'=w_{\alpha_1}w_{\alpha_2}\cdots w_{\alpha_{s-1}}w_{\alpha_s}$, we have
	\[A(w,w')=A(ww_{\alpha_1}\cdots w_{\alpha_{s-1}},w')\circ\cdots\circ A(ww_{\alpha_1},ww_{\alpha_1}w_{\alpha_2})\circ A(w,ww_{\alpha_1}).\]
\end{enumerate}
For general standard parabolic group $P=MN$ of $G$ and supercuspidal representation $\rho$ of $M$, even though Bernstein--Zelevinsky geometrical lemma says that only elements in the relative Weyl group $W_M:=N_G(M)/M$ appear in $r_P\circ Ind^G_P(\rho)$ instead of the whole double coset $P\backslash G/P$, $W_M$ is NOT a coxeter group in general. To overcome such a difficulty, we discover two novel observations in what follows.

The first observation is about the structure of the relative Weyl group $W_M$. Recall that $\Phi^0_M$ (resp. $(\Phi_M^0)^+$) is the set of those relative (resp. positive) roots which contribute reflections in $W_M$ and $W_M^0$ is the ``small'' relative Weyl group $W_M^0:=\left<w_\alpha:~\alpha\in \Phi^0_M\right>$. Notice that $w_{w.\alpha}=ww_\alpha w^{-1}$ for any $w\in W_M$ and $\alpha\in \Phi_M$, we know that $W_M$ preserves $\Phi_M^0$, whence $W_M^0\lhd W_M$ and $\Phi_M^0$ is a root system. From the previous analysis of regular principal series, it is readily to see that all previous arguments works perfectly for the ``small'' relative Weyl group $W_M^0$ which is a coxeter group.    

Analogous to the definition of Knapp--Stein $R$-group, we define 
\[W_M^1:=\big\{w\in W_M:~w.(\Phi_M^0)^+>0 \big\}. \]
It is easy to see that $W_M^1$ is a subgroup of $W_M$, and $W_M^1\cap W_M^0=\{1\}$. Then it is natural to guess that $W_M=W_M^0 W_M^1$ which is stated as the following lemma.
\begin{lem}\label{key2}
	Keep the notation as above. Then we have
	\[W_M=W_M^0\rtimes W_M^1. \]
\end{lem}   
\begin{proof}
	It suffices to show 
	\[W_M=W_M^0. W_M^1. \]
	Recall that $\Phi_M^0$ is a root system with the set of simple roots denoted by $\Delta_M^0$, and $W_M$ preserves $\Phi_M^0$. Therefore for $w\in W_M$, we may consider the action of $w$ on  $\Delta_M^0$. If
	\[w.\Delta_M^0>0, \qquad i.e.\quad w.\Phi_M^0>0, \]
	then by definition, we have $w\in W_M^1$. Otherwise, there exists some $\alpha\in \Delta_M^0$ such that
	\[w.\alpha<0. \]
	By the decomposition structure of $w$ as reflections as in \cite{moeglin_waldspurger_1995}, we know that 
	\[w=w_\alpha  w'\mbox{ with }\mathfrak{l}(w)>\mathfrak{l}(w'), \]
	where $\mathfrak{l}(-)$ is the length function. It is easy to see that $w'\in W_M$. So the Lemma follows by induction on the length of $w$.	 
\end{proof}
The second observation is about the structure of those ``mysterious'' intertwining operators $A(w,ww_1)$ with $w\in W_M^0$ and $w_1\in W_M^1$.

As pointed out previously, we define $S$ to be the set of those relative positive coroots $\alpha^\vee$ such that the co-rank one parabolic induction $Ind^{M_\alpha}_{P\cap M_\alpha}(\rho)$ associated to $\alpha\in \Phi^0_M$ is reducible. Notice that $W_M\neq W_M^0$in general, thus the following key Lemma is needed to claim a similar result as Theorem \ref{rps} for regular generalized principal series. 
\begin{lem}\label{key1}Keep the notions as above.
	For any $w\in W_M^0$ and any $w_1\in W_M^1$, we have that
	\[A(w,ww_1):~Ind^G_P(\rho^w)\longrightarrow Ind^G_P(\rho^{ww_1})\mbox{ is an isomorphism.} \]
\end{lem}
\begin{proof}
	it reduces to show 
	\[Ind^G_P(\rho)\simeq Ind^G_P(\rho^{w_1}), \]
	which follows from the associativity property of intertwining operators (please refer to \cite[IV.3.(4)]{waldspurger2003formule} for the notions). To be precise, up to non-zero scalar, the non-trivial intertwining operator
	\[A: ~Ind_P^G(\rho)\longrightarrow Ind_P^G(\rho^{w_1}) \]
	is equal to 
	\[J_{P|w_1^{-1}Pw_1}(\rho^{w_1})\circ\lambda(w_1):~Ind_P^G(\rho)\longrightarrow Ind^G_{w_1^{-1}Pw_1}(\rho^{w_1})\longrightarrow Ind_P^G(\rho^{w_1}).  \]
	By \cite[IV.3.(4)]{waldspurger2003formule}, we have
	\[J_{P|w_1^{-1}Pw_1}(\rho^{w_1})J_{w_1^{-1}Pw_1|P}(\rho)=\prod j_\alpha(\rho) J_{P|P}(\rho), \]
	where $\alpha$ runs over $\Phi_M(P)\bigcap \Phi_M(\overline{w_1^{-1}Pw_1})$ with $\overline{w_1^{-1}Pw_1}$ the opposite parabolic subgroup of $w_1^{-1}Pw_1$. Notice that 
	\[w_1.C^+_M=C^+_M, \]
	so we have
	\[\Phi_M(P)\bigcap \Phi_M(\overline{w_1^{-1}Pw_1})\bigcap (S\bigcup -S)=\emptyset .\]
	Thus, in view of \cite[Corollary 1.8]{silberger1980special},
	\[\prod j_\alpha(\rho) J_{P|P}(\rho)=\prod j_\alpha(\rho)\neq 0,~\infty. \]
	Whence $A$ is an isomorphism. 	
\end{proof}
Combining the previous analysis and our key Lemma \ref{key1}, we could now claim our Rodier type structure theorem for the general case as follows:
\begin{thm}(Rodier Type Structure Theorem)\label{mthm}
	Keep the notation as previous. The constituents $\pi_\Gamma$ of the regular generalized principal series $Ind^G_P(\rho)$ are parameterized by the connected components $\Gamma$ of $$^{0}\mathfrak{a}_M^\star-\bigcup_{\alpha^\vee\in S}Ker(\alpha^\vee)$$
	satisfying the following property:
	
	the Jacquet module $r_P(\pi_\Gamma)$ of $\pi_\Gamma$ with respect to $P$ is equal to 
	\[\bigoplus_{wC^+_M\subset \Gamma}\rho^w. \]	
\end{thm}
A direct corollary of Rodier type structure theorm, i.e. Theorem \ref{mthm} is the well-known Harish-Chandra--Silberger's irreducibility criterion for regular generalized principal series $Ind^G_P(\rho)$ as follows:
\begin{thm}(see \cite[Theorem 5.4.3.7]{silberger2015introduction})\label{irredcrit}
	If $\rho$ is a regular supercuspidal representation of the Levi subgroup $M$ of $P=MN$ in $G$, then the following are equivalent
	\begin{enumerate}[(i)]
		\item $Ind^G_P(\rho)$ is irreducible.
		\item $S$ is an empty set, i.e. no co-rank one reducibility.
	\end{enumerate}
\end{thm} 
Well, to sum up, we would like to emphasis that the previous argument works in a broad sense if the following two analogous ingredients exist in general
\begin{enumerate}[(i)]
	\item Bernstein--Zelevinsky geometrical lemma (cf. \cite{bernstein1977induced,casselman1995introduction}).
	\item Harish-Chandra Plancherel formula theory, especially intertwining operator theory (cf. \cite{waldspurger2003formule}).
\end{enumerate}
A direct example is that the finite central covering group $\widetilde{G}$ of $G$ enjoys those properties listed as above (cf. \cite{BJ,luo2017R}). Thus we know that
\begin{thm}\label{covering}
	Rodier type structure theorem, i.e. Theorem \ref{mthm} and Theorem \ref{irredcrit} hold for finite central covering group $\widetilde{G}$.
\end{thm}
Another direct corollary of the above structure theorems, i.e. Theorem \ref{mthm} and Theorem \ref{covering}, is as follows:
\begin{cor}\label{univirred}
	Keep the notions as before. If all co-rank one reducibility conditions lie in a Levi subgroup $L$ of a parabolic subgroup $Q=LV$ in $G$, then we have
	\[Ind^G_Q(\sigma)\mbox{ is always irreducible for any }\sigma\in JH(Ind^L_{L\cap P}(\rho)). \]
\end{cor}
Indeed, such a natural universal irreducibility structure is a special case of a general phenomenon \cite[Theorem 2.1]{luo2018C}.
	
\section{linearly independence of co-rank one reducibility conditions}
Recall that $P=MN$ is a standard parabolic subgroup of $G$, $\rho$ is a regular supercuspidal representation of $M$, $\Phi_M$ (resp. $\Phi_M^\vee$) is the set of relative reduced roots (resp. coroots) determined by $P$ and $S$ is the set of those positive relative coroots in $\Phi_M^\vee$ such that the associated co-rank one inductions are reducible. Denote by $\rho_0$ the unitary part of $\rho$, and by $\omega_\rho$ the real unramified part of the central character of $\rho$, i.e. $\omega_\rho\in \mathfrak{a}_M^*$.

As an essential input for the determination of square-integrable/tempered constituents in $JH(Ind^G_P(\rho))$, we need to prove the following claim (cf. \cite[Proposition 3]{rodier1981decomposition} for principal series of split groups) which says that  
\begin{thm}\label{linearindep}
	Keep the notions as before. The set $S$ is linearly independent.
\end{thm}
Before turning to the general proof, we first serve you some observations which play a key role in what follows.

The first observation is about $W_{\rho_0}:=\big\{w\in W_M:~w.\rho_0=\rho_0 \big\}$:

Let $W_S$ be the subgroup of $W_M:=N_G(M)/M$ generated by $S$, i.e.
\[W_S:=\left<w_\alpha:~\alpha^\vee\in S \right>. \] 
Denote by $\Phi_S$ the sub-coroot system of $\Phi_M$ generated by $S$, i.e. 
\[\Phi_S:=\big\{w.\alpha^\vee:~\alpha^\vee\in S,w\in W_S  \big\}. \]
It is easy to see that $W_S$ is the Weyl group of $\Phi_S$. Notice that for any $\alpha^\vee\in S$, we have
\[w_\alpha.\rho_0=\rho_0, \]
whence $\rho_0$ is fixed by $W_S$, i.e.
\[W_S\subset W_{\rho_0}. \] 

The second observation is about the uniformity of the unique pole of the co-rank one Plancherel measure $\mu_\alpha(s,\rho_0)$ (please refer to \cite{waldspurger2003formule,silberger1980special} for the notion).

It is well-known that a non-tempered co-rank one induction $Ind^{M_\alpha}_{P\cap M_\alpha}(\rho)$ associated to $\alpha^\vee\in \Phi_M$ is reducible if and only if the co-rank one Plancherel measure $\mu_\alpha(s,\rho_0)$ has a pole at $\rho$, i.e. $(w_{\rho})_\alpha:=\omega_\rho\circ H_{\alpha^\vee}=|\cdot|^{\pm s_0}$ for a unique $s_0> 0$ (cf. \cite{silberger1980special}). Such an $s_0$ is uniquely determined by $\alpha$ and $\rho_0$. Notice that 
\[Ind^{M_\alpha}_{P\cap M_\alpha}(\rho_0)\simeq Ind^{M_\alpha^w}_{P\cap M_\alpha^w}(\rho_0^w) \]
for any $w\in W_M$, we have $\mu_\alpha(s,\rho_0)=\mu_{w.\alpha}(s,\rho_0^w)$ for any $\alpha^\vee\in \Phi_M$. In particular, for $w\in W_S$, i.e. $w.\rho_o=\rho_0$, we know that \[\mu_{\alpha}(s,\rho_0)=\mu_{w.\alpha}(s,\rho_0),\]
which is to say that
\[\mbox{ Such an $s_0$ is an invariant under a $W_S$-orbit in $\Phi_S$.}\tag{$UF$}\]
In view of those two observations, in order to prove Theorem \ref{linearindep}, it suffices to show that
\[\mbox{there exists a coroot $\beta^\vee\in \Phi_S$ such that $(\omega_\rho)_\beta=|\cdot|^0=1$}\tag{$CC$}\] 
if the set $S$ is not linearly independent under the assumption that the root system $\Phi_S$ is irreducible (cf. \cite[Proposition 3]{rodier1981decomposition}). Note that for irreducible root system $\Phi_S$, an easy calculation shows that $W_S$-orbits are determined uniquely by the lengths of roots, which is at most of two, i.e.
\[\mbox{Types $A_n, D_n$ and $E_n$: Single $W_S$-orbit; Types $B_n,C_n,G_2$ and $F_4$: Two $W_S$-orbits.} \]
Denote by $S^+$ the set of those coroot $\alpha^\vee$ in $S\bigcup -S$ such that $(\omega_\rho)_\alpha=|\cdot|^{s_0}$ with $s_0>0$.
Notice that if there are two coroots $\alpha_1^\vee$ and $\alpha_2^\vee\in S^+$ of same length such that $\left<\alpha_1,\alpha_2^\vee\right>>0$, then $w_{\alpha_2}.\alpha_1^\vee=\alpha_1^\vee-\alpha_2^\vee\in \Phi_S$, which in turn says that $(\omega_\rho)_\beta=1$. Via $(CC)$, we get a contradiction with the regular condition of $\rho$. Thus
\[\left<\alpha,\beta^\vee \right>\leq 0 \mbox{ for any }\alpha^\vee\neq \beta^\vee\in S^+ \mbox{ of same length}.\tag{$OB$}  \]
Now we could begin our proof of the claim that $S^+$ is linearly independent case-by-case in terms of irreducible root types of $\Phi_S$ as follows:
\begin{proof}
	\underline{Types $A_n, D_n$ and $E_n$}: In such a case, $\Phi_S$ has a single $W_S$-orbit. Then $(UF)$ says that there exists $s_0>0$ such that for any $\alpha^\vee\in \Phi_S$, $Ind^{M_\alpha}_{P\cap M_\alpha}(\rho)$ is reducible if and only if $(\omega_\rho)_\alpha=|\cdot|^{\pm s_0}$. In view of $(CC)$, we know that $\omega_\rho$ is a regular vector w.r.t our root system $\Phi_S$. Thus $S^+=\{\alpha^\vee\in \Phi_S:~\left<\omega_\rho,\alpha^\vee\right>=s_0\mbox{ is minimal positive} \}$. For such $S^+$, it is well-known that $S^+$ is a base of the root system $\Phi_S$ (for example see \cite[Proposition 19.7]{magerty2012}). Whence our claim holds via $(CC)$.
	
	\underline{Types $B_n, C_n, G_2$ and $F_4$}: In this case, $\Phi_S$ has two $W_S$-orbits given by the length of roots. Set those two poles of our co-rank one Plancherel measure to be at $s_0>0$ and $t_0>0$. Because of $(UF)$, we denote $S^+_{s_0}$ to be the subset of $S^+$ consisting of those coroots $\alpha^\vee\in S^+$ such that $(\omega_\rho)_\alpha=|\cdot|^{s_0}$, and denote $S^+_{t_0}$ to be the subset of $S^+$ consisting of those coroots $\beta^\vee\in S^+$ such that $(\omega_\rho)_\beta=|\cdot|^{t_0}$. Applying the same argument as in the simply-laced case, i.e. \underline{Types $A_n, D_n$ and $E_n$}, we know that $S^+_{s_0}$ and $S^+_{t_0}$ are linearly independent respectively, we also know that $S^+$ is linearly independent if $s_0=t_0$. So it is reduced to deal with the case $s_0\neq t_0$. 
	
	For types $B_n$ and $C_n$, without loss of generality, we assume that elements in $S^+_{s_0}$ are of length 2. Observe that for any two different elements $\beta^\vee_1$ and $\beta^\vee_2$ in $S^+_{t_0}$, we have $\left<\omega_\rho,\beta^\vee_1-\beta^\vee_2\right>=0$ and $c(\beta^\vee_1-\beta^\vee_2)\in \Phi_S$ for some $c=1$ or $\frac{1}{2}$, which gives a contradiction with $(CC)$. So $S^+_{t_0}$ contains at most one element. Thus $\#S^+_{s_0}\geq n-1$. After conjugating by a Weyl element in $W_S$, an easy analysis shows that $\#S^+_{s_0}=n-1$ and   
	\[S^+_{s_0}=\bigg\{e_1-e_2,e_2-e_3,\cdots, e_{n-1}-e_{n} \bigg\}. \]
	Whence our claim holds. 
	
	For Type $G_2$, we set \[\Phi_S=\pm \bigg\{\alpha,\beta,\alpha+\beta,2\alpha+\beta,3\alpha+\beta,3\alpha+2\beta \bigg\},\] and set \[\mbox{$\left<\omega_\rho,\alpha\right>=a_0>0$ and $\left<\omega_\rho,\beta\right>=b_0>0$.}\] It is easy to check that $\#S^+_{s_0}=\#S^+_{t_0}=1$. Thus our claim holds.
	
	For Type $F_4$, we set all the roots of $\Phi_S$ to be:
	\begin{align*}
	 &\mbox{24 roots by }\alpha_i:=\big(\pm 1,\pm 1,0,0\big),\mbox{ permuting coordinate positions};\\
	 &\mbox{8 roots by }\beta_j:=\big(\pm 1,0,0,0\big),\mbox{ permuting coordinate positions};~and \\
	 &\mbox{16 roots by }\gamma_k:=\left(\pm\frac{1}{2},\pm\frac{1}{2},\pm\frac{1}{2},\pm\frac{1}{2}\right).
	\end{align*}
	It is easy to see that there are at most one $\beta_j$ in $S^+$. Otherwise $\left<\omega_\rho,\beta_1-\beta_2\right>=0$ contradicts $(CC)$. We set $S^+_{s_0}$ to be the subset of those coroots in $S^+$ which are of length 2. Consider the subroot system formed by roots of length 1, we know that it is of type $D_4$, thus $S^+_{t_0}$ could only be a subset of 
	\[\bigg\{\frac{1}{2}(1,1,1,1),\frac{1}{2}(1,-1,-1,1),\frac{1}{2}(1,-1,1,-1), (-1,0,0,0) \bigg\}, \]
	after conjugating by a Weyl element. Observe that applying subtraction on any two roots in 
	\[\bigg\{\frac{1}{2}(1,1,1,1),\frac{1}{2}(1,-1,-1,1),\frac{1}{2}(1,-1,1,-1) \bigg\} \]
	produces a root in $\Phi_S$, then $(CC)$ says that there are at most one $$\beta^\vee\in \bigg\{\frac{1}{2}(1,1,1,1),\frac{1}{2}(1,-1,-1,1),\frac{1}{2}(1,-1,1,-1) \bigg\}$$ which belongs to $S^+_{t_0}$. Thus $\#S^+_{t_0}\leq 2$. On the other hand, consider the subroot system formed by roots of length 2, we know that it is also of type $D_4$, thus $S^+_{s_0}$ could only be a subset of
	\[\bigg\{(1,-1,0,0),(0,1,-1,0),(0,0,1,-1),(0,0,1,1) \bigg\} \]
	after conjugating by a Weyl element. Observe that \[\mbox{$(0,0,1,-1)-(0,0,1,1)=2(0,0,0,-1)$ and $(1,-1,0,0)-(0,0,1,\pm 1)\in \Phi_S$,}\] thus $(CC)$ says that $\#S^+_{s_0}\leq 2$ and it is a subset of \[\bigg\{(1,-1,0,0),(0,1,-1,0) \bigg\}\] after conjugating by a Weyl element. But the dimension of $\Phi_S$ is 4, whence $S^+=S^+_{s_0}\bigcup S^+_{t_0}$ is linearly independent.
\end{proof}
\begin{rem}
	From the above arguments, Theorem \ref{linearindep} holds also for finite central covering groups once a covering group version of Silberger's main theorem in \cite{silberger1980special} is established.
\end{rem}

Let $\iota$ be the rank of the center of the Levi subgroup $M$ of $P=MN$ in $G$, then
\begin{cor}
	Keep the notions as before. The length of the regular generalized principal series $Ind^G_P(\rho)$ is equal to $2^{\#S}$ which is at most $2^\iota$.
\end{cor}
\section{square-integrability}
In this section, we would like to investigate the discreteness/temperedness property of the subquotients of the regular generalized principal series $Ind^G_P(\rho)$. For simplicity, we may assume that the real unramified part $\omega_\rho$ of the central character of $\rho$ lies in $\bar{\mathfrak{a}}_M^{*+}$, and we define $S$ to be the set of those relative positive coroots $\alpha^\vee$ such that the co-rank one parabolic induction $Ind^{M_\alpha}_{P\cap M_\alpha}(\rho)$ associated to $\alpha\in \Phi^0_M$ is reducible. Among the connected components which index the constituents in $JH(Ind^G_P(\rho))$ (see Theorem \ref{mthm}), there exists a distinguished one
\[\Gamma_+=\bigcap_{\alpha^\vee\in S}(\alpha^\vee)^{-1}(\mathbb{R}^+) \]
which plays a key role in what follows.

To start, let us first state a necessary condition concerning Casselman's square-integrability criterion as follows:
\begin{lem}\label{dsn}
	For generalized principal series $Ind^G_P(\rho)$ with $\rho$ regular supercuspidal representation of the Levi subgroup $M$ of $P=MN$ in $G$, let $S$ be the set of relative positive coroots such that the associated co-rank one parabolic inductions are reducible. Then there exists square-integrable $\pi\in JH(Ind^G_P(\rho))$ only if 
	\[\mathfrak{a}_M^*=~^{0}\mathfrak{a}_M^*=Span_\mathbb{R}\{\alpha:~\alpha^\vee\in S \}. \] 
\end{lem}
\begin{proof}
	This follows from Corollary \ref{univirred} and the fact that a full induced representation is not square-integrable. One may also refer to a theorem of Silberger (cf. \cite[Theorem 3.9.1]{silberger1981discrete}).
\end{proof}	
Given Lemma \ref{dsn}, we could now reinterpret Casselman's criterion under our setting which says that
\begin{prop}\label{ds}
Keep the notions as before. $JH(Ind^G_P(\rho))$ contains at most one square-integrable constituent. Moreover the representation $\pi_\Gamma$ is square-integrable if and only if \[\mbox{$^{0}\mathfrak{a}_M^*=\mathfrak{a}_M^*=Span_{\mathbb{R}}\{\alpha:~\alpha^\vee\in S \}$ and $|\omega_\rho^w(H_{\alpha^\vee}(\mathfrak{w}))|<1$} \] for any root $\alpha\in \bar{C}^+_M$ and any $w\in W_M$ such that $wC^+_M\subset \Gamma$, or equivalently, $$|\omega_\rho(H_{\alpha^\vee}(\mathfrak{w}))|<1,~\forall \alpha\in\Phi_M\cap \bar{\Gamma}.$$
That is to say $\Gamma=\Gamma_+$.
\end{prop}
\begin{proof}
	The argument of the uniqueness claim is as follows:
	
	Otherwise, if $\pi_{\Gamma_1}$ and $\pi_{\Gamma_2}$ are square-integrable, then there exists a wall $Ker~\beta^\vee$, with $\beta^\vee\in S$, separates $\Gamma_1$ and $\Gamma_2$. Hence $\bar{\Gamma}_1\cup \bar{\Gamma}_2$ contains $\bigcap\limits_{\beta^\vee\neq \alpha^\vee\in S}Ker~\alpha^\vee$ by Theorem \ref{linearindep}, which in turn says that there exits $0\neq w\in \mathfrak{a}^\star_M$ such that $\big\{w,w^{-1}\big\}\subset\bar{\Gamma}_1\cup \bar{\Gamma}_2$. Contradiction. 
	
	The only if part follows from Lemma \ref{dsn} and the definition of $\omega_\rho\in \bar{\mathfrak{a}}^{*+}_M$. 
	
	The if part follows from a direct check, one may also refer to a general theorem (cf. \cite[Corollary 8.7]{heiermann2004decomposition}).	
\end{proof}

In view of Corollary \ref{univirred}, almost the same reinterpretation works for Casselman's temperedness criterion (please refer to \cite[Proposition 6]{rodier1981decomposition} for details).

\begin{prop}\label{tem}
	Keep the notions as before. $JH(Ind^G_P(\rho))$ contains at most one tempered constituent. The representation $\pi_\Gamma$ is tempered if and only if $\omega_\rho$ restricting to the subgroup $\bigcap\limits_{\alpha^\vee\in S}Ker(\alpha^\vee)$ of $M$ is unitary and $\Gamma=\Gamma_+$.
\end{prop}
\begin{proof}
	The same type argument as above works (cf. \cite[Proposition 6]{rodier1981decomposition}).
\end{proof} 
\begin{rem}
	All the above claims also hold for finite central covering groups.
\end{rem}
\section{genericity}
Note that there exists at most one generic constituent guaranteed by Rodier's hereditary theorem (see \cite[Theorem 4]{rodiermodeles}) which has nothing to do the linearly independence property of co-rank one reducibility conditions. In this section, assume that $G$ is a connected quasi-split reductive group, and $\rho$ is a regular supercuspidal representation of the Levi subgroup $M$ of a standard parabolic subgroup $P=MN$ in $G$, we would like to give a characterization of the unique generic constituent of the regular generalized principal series $Ind^G_P(\rho)$. Given such a characterization, we serve you an easy and intuitive proof of Casselman--Shahidi's main theorem (i.e. \cite[Theorem 1]{casselman1998irreducibility}) on their standard module conjecture/generalized injectivity conjecture for regular generalized principal series.

Recall that $B=TU$ is a fixed Borel subgroup of $G$ with $U$ a maximal unipotent subgroup of $G$, $\theta$ is a generic character of $U$ and $\mathcal{W}_\theta:=Ind^G_U(\theta)$ is the Whittaker function space. As in \cite[Section 3]{shahidi1990proof}, we assume that the generic character $\theta$ of $U$ and the longest Weyl element $w_0^G$ in $W$ are compatible. Denote by $\theta_M$ the generic character $\theta$ of $U$ restricting to $N$ which is compatible with the longest Weyl element $w_0^M$ in $W^M$. In what follows, we always assume that the irreducible admissible representation $\rho$ of $M$ is {\bf regular, $\theta_M$-generic and supercuspidal}. For simplicity, we may also assume that the real unramified part $\omega_\rho$ of the central character of $\rho$ lies in $\bar{\mathfrak{a}}_M^{*+}$ the closure of $\mathfrak{a}_M^{*+}$.

Indeed, through a clear understanding of the proof of our Rodier type structure theorem, i.e. Theorem \ref{mthm}, it is readily to see that the problem of sorting out the generic constituent is reduced to determine the genericity of the constituents of those co-rank one inductions $Ind^{M_\alpha}_{P\cap M_\alpha}(\rho^w)$, where the pairs $(\alpha,w)$ are those, $\alpha\in \Delta_M^0$ and $w\in W_M^0$, satisfying the following conditions 
\begin{enumerate}[(i)]
	\item $w.\alpha$ is a relative positive root, i.e. $w.\alpha\in (\Phi^0_M)^+$. 
	\item The co-rank one induction $Ind^{M_\alpha}_{P\cap M_\alpha}(\rho^w)$ is reducible.
\end{enumerate}
The above conditions are equivalent to saying that $w.\alpha^\vee\in S$.

Let $A(w,ww_\alpha)$ be the unique, up to scalar, non-zero intertwining map in
\[Hom_G(Ind^G_P(\rho^w),~Ind^G_P(\rho^{ww_\alpha}))\simeq \mathbb{C} \]
which has the induction by stage property, i.e. the following diagram commutes:
\[\xymatrix{Ind^G_{P}(\rho^{w})\ar[r]^{A(w,ww_\alpha)}\ar@{=}[d]&Ind^G_{P}(\rho^{ww_{\alpha}})\ar@{=}[d]\\
	Ind^G_{P_\alpha}\circ Ind^{M_\alpha }_{P\cap M_\alpha}(\rho^{w})\ar[r]^-{Ind(A)}&Ind^G_{P_\alpha}\circ Ind^{M_\alpha }_{P\cap M_\alpha}(\rho^{ww_{\alpha}}),       } \]
where $P_\alpha=M_\alpha N_\alpha$ is the co-rank one parabolic subgroup associated to the relative simple root $\alpha\in \Delta^0_M$ (cf. \cite{silberger2015introduction}). Notice that $Ind^{M_\alpha}_{P\cap M_\alpha}(\rho^w)$ is a standard module by the definition of $\rho$ and the choice of the pair $(\alpha,w)$, the Langlands--Shahidi theory says that (cf. \cite{shahidi1990proof})
\[Ker(A(w,ww_\alpha))\mbox{ is $\theta$-generic}. \]
Recall that $(SB)$ in the previous section says that the Jacquet module $r_P(Ker(A(w,ww_\alpha)))$ of $Ker(A(w,ww_\alpha))$ is equal to
\[r_P(Ker(A(w,ww_\alpha)))=\bigoplus_{w''}\rho^{w''}, \]
where $w''$ runs over those relative Weyl elements in $W_M$ such that 
\[wC_M^+\mbox{ and }w''C_M^+\mbox{ are on the same side of }Ker(w.\alpha^\vee). \]
This is equivalent to saying that
\[\left<w''.\beta, w.\alpha^\vee \right>>0, \]
as 
\[\left<w\beta,w.\alpha^\vee \right>=\left<\beta,\alpha^\vee \right>>0\]
for any $\beta\in C_M^+$. Then our Rodier type structure theorem, i.e. Theorem \ref{mthm} implies that 
\begin{thm}\label{gene}
	Keep the notions as previous. Set $$\Gamma_+:=\bigcap\limits_{\alpha^\vee\in S}(\alpha^\vee)^{-1}(\mathbb{R}^+).$$ Then we have
	\[\pi_{\Gamma_+}\mbox{ is the unique generic constituent in }JH(Ind^G_P(\rho)), \]
	here $\pi_{\Gamma_+}$ stands for the corresponding constituent given by $\Gamma_+$ in Theorem \ref{mthm}.
\end{thm}
Now we could serve you a simple intuitive proof of Casselman--Shahidi's main theorem \cite[Theorem 1]{casselman1998irreducibility} which is a corollary of the above Theorem \ref{gene} as follows.

Recall that the regular generalized principal series $Ind^G_P(\rho)$ has a unique irreducible subrepresentation $\pi_+$. Moveover, $\pi_{+}$ is uniquely determined by the fact that its Jacquet module with respect to $P$ contains $\rho$, i.e.
\[\rho\in r_P(\pi_{+}). \] 
As the co-rank one reducibility set $S$ consists of relative positive coroots, then for any $\alpha^\vee\in S$ and $\beta\in C_M^+$, we have
\[\left<\beta,\alpha^\vee\right>>0. \]
On the other hand, by convention, the real unramified part $\omega_\rho$ of the central character of $\rho$ lies in $\bar{\mathfrak{a}}_M^{*+}$. Thus 
\[\rho\in r_P(\pi_{\Gamma_+}), \]
whence $\pi_+=\pi_{\Gamma_+}$ is the generic constituent in $JH(Ind^G_P(\rho))$, especially Casselman--Shahidi's main theorem (cf. \cite[Theorem 1]{casselman1998irreducibility}) holds as follows:
\begin{thm}(see \cite[Theorem 1]{casselman1998irreducibility})\label{generic}
	For the standard representation $Ind^G_P(\rho)$ with $\rho$ a generic supercuspidal representation of the Levi subgroup $M$ of a standard parabolic subgroup $P=MN$ in a connected quasi-split reductive group $G$, we have that 
	\[\mbox{ the unique generic subquotient of }Ind^G_P(\rho) \mbox{ is a subrepresentation.} \] 
\end{thm}

\section{aubert duality}
Motivated by Bernstein's unitarity conjecture on Aubert duality, and Hiraga's conjecture on the description of Aubert duality in terms of Arthur parameters (cf. \cite{hiraga2004functoriality}), we would like to investigate how Aubert duality acts on $JH(Ind^G_P(\rho))$ with $\rho$ a regular supercuspidal representation of the Levi subgroup $M$ of $P=MN$ in $G$. Please refer to \cite{aubert1995dualite} for its definitions and properties.

Note that the Aubert duality $D$ commutes with parabolic induction, thus for a relative simple root $\alpha$ in $\Phi_M^0$, $$D(Ker(A(w,ww_\alpha)))=Im(A(w,ww_\alpha))$$ provided that the unique, up to scalar, intertwining map $A(w,ww_\alpha)$ is not an isomorphism in
\[Hom_G(Ind^G_P(\rho^w),~Ind^G_P(\rho^{ww_\alpha})\simeq \mathbb{C}. \]
Therefore $$D(\pi_\Gamma)=\pi_{-\Gamma},$$ where $\pi_\Gamma$ is the constituent in $JH(Ind^G_P(\rho))$ corresponding to the component $\Gamma$ in $$^{0}\mathfrak{a}_M^\star-\bigcup_{\alpha^\vee\in S}Ker(\alpha^\vee)$$ given by Theorem \ref{mthm}. Thus
\begin{cor}
	Under Aubert duality, there is no fixed point within the constituents of regular generalized principal series.
\end{cor}

\bibliographystyle{amsalpha}
\bibliography{ref}

			
\end{document}